\documentclass[12pt]{amsart}
\usepackage{amssymb,amsthm,amsmath, amsfonts}
\usepackage{graphicx}
\usepackage{fullpage}

\usepackage{xcolor}
\definecolor{vry}{RGB}{253, 231, 37}

\definecolor{vrg}{RGB}{94,201,98}

\definecolor{vrdg}{RGB}{33, 145, 140}

\definecolor{vrb}{RGB}{59,82,139}

\definecolor{vrp}{RGB}{68,1,84}

\definecolor{vro}{RGB}{249,142,9}

\definecolor{vrr}{RGB}{188,55,84}

\definecolor{vrnb}{RGB}{13,8,135}

\usepackage{hyperref}
\hypersetup{colorlinks=true, linkcolor=vrnb, citecolor=vrp, filecolor=vrr, urlcolor=vrr}

\theoremstyle{plain}
\newtheorem{theorem}{Theorem}[section]
\newtheorem{proposition}[theorem]{Proposition}
\newtheorem{lemma}[theorem]{Lemma}
\newtheorem{remark}[theorem]{Remark}

\newtheorem{corollary}[theorem]{Corollary}

\newtheorem{definition}[theorem]{Definition}

\newcommand{\veps}{\varepsilon} 

\newcommand{\Des}{\operatorname{Des}}

\newcommand{\des}{\operatorname{des}}
\newcommand{\Asc}{\operatorname{Asc}}
\newcommand{\asc}{\operatorname{asc}}
\newcommand{\Vl}{\operatorname{Vl}}

\newcommand{\vl}{\operatorname{vl}}
\newcommand{\Pk}{\operatorname{Pk}}
\newcommand{\pk}{\operatorname{pk}}

\newcommand{\x}{\mathbf{x}}
\newcommand{\E}{\mathbb E}
\newcommand{\Pbb}{\mathbb P}
\newcommand{\1}{\mathbf 1}

\newcommand{\Var}{\mathbb{V}}

\newcommand{\dd}{\delta}
\newcommand{\inv}{\operatorname{inv}}
\newcommand{\wt}{\operatorname{wt}}

\title{Analysis of the asymmetric shelf shuffle}
\author[R. Tripathi]{Raghavendra Tripathi }
\address{Raghavendra Tripathi. Division of Science, NYU Abu Dhabi, Abu Dhabi, UAE}
\email{r.tripathi@nyu.edu}
\thanks{I thank Prof. Persi Diaconis for his encouraging comments and for suggesting several useful references. I also thank Prof. Jason Fulman for answering some questions regarding~\cite{DiaconisFulmanHolmes2013}.}
\begin{document}

\begin{abstract}
In an asymmetric shelf shuffle, a deck of $n$ cards is dealt sequentially from the bottom and assigned one of the $m$ shelves uniformly at random. The card is placed at the top of the assigned shelf with probability $p$, and at the bottom of the assigned shelf with probability $(1-p)$. Analysis of the shelf shuffle has gained much attention recently, and the case $p=1/2$ was first treated by Diaconis--Fulman--Holmes [Ann. Appl. Prob. 23 (2013), no. 4, 1692--1720]. In this paper, we extend the analysis of the shelf shuffle to general $p\in (0, 1)$. In particular, we study the distribution of cycles, cycle lengths, number of descents, number of valleys, number of inversions, and the RSK shape of a permutation obtained from an asymmetric shelf shuffle. Our results extend the analysis of Diaconis--Fulman--Holmes to arbitrary $p$. Furthermore, our analysis of the distribution of descents and inversions is new even for $p=1/2$.
\end{abstract}


\maketitle
\section{Introduction}\label{sec:Intro}
The analysis of card shuffling has a rich and deep history. The earliest works in this direction can be traced back to Hadamard~\cite{Hadamard} and Poincar\'e~\cite{Poincare1912}. The modern era of the analysis of card shuffling began with the rigorous analysis of the properties of the riffle-shuffle~\cite{Diaconis81Generating, Aldous83, aldous1986shuffling, bayer1992trailing, diaconis1995riffle}. We refer the reader to the recent book by Diaconis and Fulman~\cite{DiaconisFulman2023Shuffling} for a modern and comprehensive account of the mathematics of card shuffling, and~\cite[Chapter 4]{Diaconis88Group} and references therein. 

There are several mathematically interesting and deep aspects of card shuffling. The initial works in this area focused on the mixing properties of card shuffling schemes~\cite{chen2025cutoff, Nestoridi25Cutoff, Sellke22Cutoff}. This situates the card shuffling within the broader framework of mixing time and the cutoff phenomenon for Markov chains~\cite{salez2025modern}. Another aspect of card shuffling is card-guessing games with feedback. A player guesses the cards from a shuffled deck sequentially until the deck is exhausted. The player is given feedback after each guess (for example, being shown the last guessed card). The main questions of interest are the optimal strategy that maximizes the expected number of correct guesses and the number of correct guesses one can make under the optimal strategy. Finally, we mention another aspect of card shuffling, which is the focus of this paper. One is often interested in studying the statistics of the random permutations induced by a card shuffling scheme. On the one hand, the behaviour of these statistics can be useful in distinguishing a card shuffling from perfect randomization. Indeed, the mixing of particular statistics is often used to get a lower bound on the mixing time. On the other hand, in practical situations, one may be interested in sufficiently randomizing a given statistic, which can happen much earlier, before the whole deck mixes~\cite{clay2025limit,diaconis1995riffle}.

In this paper, we study several statistics of the asymmetric shelf shuffle, which is a generalization of the \emph{shelf shuffle} first studied by Diaconis, Fulman, and Holmes~\cite{DiaconisFulmanHolmes2013}. We refer to~\cite{DiaconisFulmanHolmes2013} for the fascinating backstory of the shelf shuffle and a detailed analysis of the shelf shuffle. We begin with the definition of asymmetric shelf shuffle introduced in~\cite{kuba}. 

\begin{definition}[Asymmetric shelf-shuffle]
Let $m\in \mathbb{N}$ be a natural number and $p\in (0, 1)$. An $(m, p)$-shelf shuffler has $m$ labelled shelves. A deck of cards labelled $1, \ldots, n$ is dealt from the bottom sequentially and assigned a shelf independently and uniformly at random. Each card is placed at the top or bottom of the existing pile, on the assigned shelf, with probability $p$ and $1-p$, respectively, independently of everything else. The shuffled deck is then formed by joining the piles in shelf order. 
\end{definition}

When $p=1/2$, this corresponds to the shelf shuffle of~\cite{DiaconisFulmanHolmes2013}, which we call the \emph{symmetric shelf shuffle}. We reserve the term shelf shuffle for the more general (asymmetric) shelf shuffle. We use $\nu_{n, m}^{(p)}$ to denote the law of the permutation induced by an $(m, p)$-shelf shuffle on a deck of $n$ cards. Throughout this paper, we will assume that $p\in (0, 1)$ and write $q=1-p$.

In a recent work, Chen and Ottolini~\cite{chen2025cutoff} proved the cutoff for the symmetric shelf shuffle, and Clay~\cite{clay2025limit} studied the limiting distribution for the descents and inversions in the symmetric shelf shuffle. Several properties of card guessing games after a symmetric shelf shuffle have been studied in~\cite{clay2025guessing, tripathi2026position}. The asymmetric shelf shuffle was introduced in~\cite{kuba}, where the authors also carried out a detailed analysis of the score in the complete feedback game after one round of shelf shuffle. For an arbitrary $p\in (0, 1)$, Tripathi obtained~\cite{Trip26Law} an explicit formula for $\nu_{n, m}^{(p)}$ and proved the cutoff (in terms of the number of shelves $m$) for the asymmetric shelf shuffle. In this paper, we carry this line of enquiry forward and study the statistics of the asymmetric shelf shuffle, including the distribution of cycle types, the distribution of RSK shape, and the distribution of descents, inversions, and inverse descents.

\subsection{Setup, notations, and background}
For a permutation $\pi=(\pi_1, \ldots, \pi_n)\in S_n$, we define the descent set and the valley set, respectively, as
\[
\Des(\pi) = \{i\in [n-1]: \pi_{i}>\pi_{i+1}\}, \quad \Vl(\pi) = \{2\leq i\leq n-1: \pi_{i-1}>\pi_i<\pi_{i+1}\}\;.
\]
We also define the ascent set $\Asc(\pi) := [n-1]\setminus\Des(\pi)$. We write $\des(\pi) = |\Des(\pi)|, \asc(\pi) = |\Asc(\pi)|$ and $ \vl(\pi) = |\Vl(\pi)|$ for the number of descents, ascents, and valleys in a permutation $\pi$, respectively. We also define the peaks in a permutation $\pi\in S_n$ as 
\[
\Pk(\pi):=\{2\leq i\leq n-1: \pi_{i-1}<\pi_i>\pi_{i+1}\}\;.
\]
We also set $\pk(\pi)=|\Pk(\pi)|$.

For a permutation $\pi\in S_n$, such that $d=\des(\pi), a=\asc(\pi)$ and $v=\vl(\pi)$, define a polynomial
\[F_\pi(u) := (1+pu)^{a-v}(1+qu)^{d-v}(1+u)^v\;.\]
The following result was shown in~\cite{Trip26Law}, and it will be repeatedly used throughout the paper. We record it for easy reference. 
\begin{theorem}[~\cite{Trip26Law}]
\label{thm:Law}
Let $n\geq 2, m\in \mathbb{N}$ and $p\in (0, 1)$. Then, for any $\pi\in S_n$, we have 
\begin{align*}
\nu_{n, m}^{(p)}(\pi)&= \frac{1}{m^n}[s^n](1+s)^{m}F_{\pi}(s)\;. 
\end{align*}
Here, and throughout this paper, we use $[u^n]P(u)$ to denote the coefficient of $u^n$ in the formal power series or polynomial $P$.
\end{theorem}
For our purposes, it will also be convenient to rewrite $\nu_{n, m}^{(p)}(\pi)$ in another way. Throughout this paper, we use $w_0$ to denote the permutation $w_0=(n,n-1,\ldots,1)$. For a permutation $\pi\in S_n$, let $\rho=w_0\pi w_0^{-1}$. Then, $\rho$ and $\pi$ have the same cycle type, as the conjugation does not change the cycle type, and it can be easily verified that
\[
\des(\rho)=\des(\pi), \qquad \pk(\rho)=\vl(\pi).
\]
Using this and the fact that $\asc(\pi)+\des(\pi)=n-1$ for every $\pi\in S_n$, we get 
\begin{align*}
F_{w_0^{-1}\rho w_0}(s) = F_{\pi}(s)&=(1+ps)^{\asc(\pi)-\vl(\pi)}
(1+qs)^{\des(\pi)-\vl(\pi)}
(1+s)^{\vl(\pi)}  \\
& = (1+ps)^{n-1}\beta(s)^{\des(\rho)}\alpha(s)^{\pk(\rho)},
\end{align*}
where 
\[
\alpha(s):=\frac{1+s}{(1+ps)(1+qs)},
\qquad
\beta(s):=\frac{1+qs}{1+ps}.
\]
This leads to the following immediate corollary of Theorem~\ref{thm:Law}, which will be repeatedly used throughout this paper.
\begin{corollary}
\label{cor:Law}
    Let $n, m, p$ be as in Theorem~\ref{thm:Law}. Let $\pi, \rho, \alpha(s), \beta(s)$ be as above. Then, 
    \[
    {\nu_{n, m}^{(p)}(\pi)} = \frac{1}{m^n}[s^n](1+s)^{m}(1+ps)^{n-1}\beta(s)^{\des(\rho)}\alpha(s)^{\pk(\rho)}\;.
    \]
\end{corollary}

\subsubsection*{Gessel and Zhuang formulas}
We also record two formulas due to Gessel and Zhuang, which are used in our derivation of the generating functions. Recall that any permutation can be written uniquely as a product of disjoint cycles (up to the ordering of these cycles). For a permutation $\pi\in S_n$, we denote the number of cycles of length $i$ by $N_i(\pi)$. And, we define 
\begin{equation}
\label{eqn:CycleStructurePoly}
    X(\pi; {\mathbf{x}}) \equiv X(\pi) :=\prod_{i\geq 1}x_i^{N_i(\pi)}\;.
\end{equation}
We now define the peak-descent polynomial as
\begin{equation}
\label{eqn:PeakDescent}
    F_n^{\pk,\des}(y,t;\mathbf{x}):=\sum_{\rho\in S_n}
y^{\pk(\rho)+1}t^{\des(\rho)+1}X(\rho;\mathbf{x}).
\end{equation}

We record the following formula~\eqref{eq:gz} from Gessel and Zhuang~\cite[Theorem 3.9]{GZ20} (for the proof see~\cite[Theorem 7.1]{GS20Plethystic}) that will be crucial in our proof. 
\begin{align}
\label{eq:gz}
&\frac{1}{1-t}
+
\frac{1}{1+y}
\sum_{n\ge 1}
\left(\frac{1+yt}{1-t}\right)^{n+1}
F_n^{\pk,\des}\left(
\frac{(1+y)^2t}{(y+t)(1+yt)},
\frac{y+t}{1+yt};\mathbf{x}
\right)z^n
\nonumber \\
&\quad =
\sum_{k\ge 0}t^k
\prod_{i\ge 1}
\exp\left\{
\sum_{r\ge 1}
\frac{(x_i z^i)^r}{ir}
\sum_{d\mid i}
\mu(d)
\bigl(k(1-(-y)^{dr})\bigr)^{i/d}
\right\},
\end{align}
where $\mu:\mathbb{N}\to \{-1, 0, 1\}$ is the standard number theoretic M\"obius function defined as 
\[
\mu(x) = \begin{cases}
    1, & x=1, \\
    (-1)^k, & \text{if }x\text{ is a product of }k\text{ distinct primes}, \\
    0, & \text{otherwise}
\end{cases}.
\]
When ${\bf x}={\bf 1} = (1,\ldots, 1)$, we define, by an abuse of notation,
\[
F_n^{\pk,\des}(y,t)\equiv F_n^{\pk,\des}(y, t; {\bf 1}) =\sum_{\rho\in S_n}
y^{\pk(\rho)+1}t^{\des(\rho)+1}\;.
\]
Gessel and Zhuang (see~\cite[Theorem 3.1]{GS20Plethystic} and~\cite[Equation 3.1]{GS20Plethystic}) provide an explicit formula for $F_n^{\pk, \des}$ which we record below.
\begin{equation}
\label{eq:GZ-run}
F_n^{\pk,\des}
\left(\frac{(1+y)^2r}{(y+r)(1+yr)},\frac{y+r}{1+yr}\right)
=
\left(\frac{(1+y)(1-r)}{1+yr}\right)^{n+1}\sum_{k\geq1}k^n r^k.
\end{equation}

The power series $\sum_{k\geq 1}k^nr^k$ is the classical Eulerian-series form: $(1-r)^{n+1}\sum_{k\geq 1}k^nr^k$ is the Eulerian polynomial. Eulerian polynomials go back to Euler's 1755 book \emph{Institutiones calculi differentialis cum eius usu in analysi finitorum ac doctrina serierum} and frequently arise in probability and combinatorics~\cite{Comtet74,Tanny73Eulerian,Kyle15}.

\subsubsection{Contribution}
For the symmetric shelf shuffle, Diaconis--Fulman--Holmes~\cite{DiaconisFulmanHolmes2013} obtained an explicit formula for the law of the symmetric shelf shuffle $\nu_{n, m}^{(1/2)}$. They also derived several generating function identities and used them to study the distribution of cycle types~\cite[Section 3.4]{DiaconisFulmanHolmes2013}, distribution of RSK shape~\cite[Section 3.5]{DiaconisFulmanHolmes2013}, and the distribution of \emph{inverse descents}, that is, the descents in the inverse permutation~\cite[Section 3.6]{DiaconisFulmanHolmes2013}. We extend all of these results to a general asymmetric shelf shuffle and obtain some additional results. 

In Section~\ref{sec:CycleStructure}, we study the cycle structure of a permutation obtained from a shelf shuffle. In Theorem~\ref{thm:cycleIndex}, we obtain the generating function for the cycle index and use it to obtain the enumerator of the expected number of $i$-cycles in Corollary~\ref{cor:expected-cycles} and the expected number of fixed points in Corollary~\ref{cor:fixedPoints}. In Corollary~\ref{cor:geometric-and-limit}, we study the convergence of the joint distribution of the number of cycles of different lengths. We also show that for any $p\in (0, 1)$, the joint distribution of the $k$ largest cycle lengths has the same $\operatorname{Poisson--Dirichlet}(1)$ limit as in the case of uniform random permutation. 

In Section~\ref{sec:Descents}, we study the number of descents in the shelf shuffle. The results in this section are new even for $p=1/2$. In Theorem~\ref{thm:GF-Descents}, we obtain the generating function of the descents and use it to obtain the first two moments of the descents in Corollary~\ref{cor:MomentsOfDescents}. In Theorem~\ref{thm:varianceAsymptotics}, we prove that the variance of the number of descents in the shelf shuffle divided by $n$ converges to $pq$ when $m=o(n)$ and converges to $1/12$ when $n=o(m)$. This significantly strengthens an observation made (communicated to us privately) by Alexander Clay. 

In Section~\ref{sec:Inversions}, we obtain a generating function for the number of inversions in Theorem~\ref{thm:inv-gf} and derive an exact formula for the mean and the variance of the number of inversions in Corollary~\ref{cor:inv-moments}. To the best of our knowledge, these results are also new even for $p=1/2$. 

Finally, in Section~\ref{sec:RSK-and-inv-descent}, we study the distribution of the RSK shape of the inverse of the permutation obtained from the shelf shuffle. This is inspired from~\cite[Section 3.5, 3.6]{DiaconisFulmanHolmes2013}. We show (Theorem~\ref{thm:rsk-shape}) that the recording tableau of the inverse of the permutation obtained from a shelf shuffle depends only on the shape of the tableau. Using this, we obtain a generating function for the descents in the inverse permutation in Theorem~\ref{thm:inverse-descent-gf} and finally obtain the first two moments of the descents in the inverse permutation in Corollary~\ref{cor:inverse-desc-moments}.

\begin{remark}
\label{rem:DFH-des}
Let us clarify that there is a minor typo in~\cite[Section 3.5, 3.6]{DiaconisFulmanHolmes2013}. The distribution of the RSK shape and the distribution of descents carried out in Diaconis--Fulman--Holmes is actually true for the \emph{inverse} of the permutation $\pi$ obtained from a symmetric shelf shuffle. In particular, the formulas for the mean and the variance in~\cite[Corollary 3.11]{DiaconisFulmanHolmes2013} are valid for the inverse permutation. To the best of our knowledge, the exact variance for the number of descents in the permutation has not been derived elsewhere.  
\end{remark}

\section{Statistics of asymmetric shelf shuffle}
\label{sec:Statistics}

\subsection{Cycle structure}
\label{sec:CycleStructure}
We recall that for a permutation $\pi$, we denote by $N_i(\pi)$ the number of cycles of length $i$ in $\pi$.  Our next result describes the multivariate generating function of the number of cycles of length $i$ in a permutation $\pi\sim \nu_{n, m}^{(p)}$. 

\begin{theorem}
\label{thm:cycleIndex}
Fix $m\geq 1$ and $p\in (0, 1)$. Let $X(\pi; {\bf x})$ be as in~\eqref{eqn:CycleStructurePoly}. Then,
\begin{equation}
\label{eq:asym-cycle-index}
Z(u, {\mathbf{x}}):=1+\sum_{n\ge 1}u^n\sum_{\pi\in S_n}\nu^{(p)}_{n,m}(\pi)X(\pi ; \mathbf{x})
=
\prod_{i\ge 1}\exp\left\{\sum_{r\ge 1}\frac{B^{(m,p)}_{i,r}}{ir}\,x_i^r u^{ir}
\right\},
\end{equation}
where 
\begin{equation}
\label{eq:B-def}
B^{(m,p)}_{i,r}:=
\sum_{d\mid i} \mu(d) m^{\,i/d-ir} \bigl(p^{dr}-(-q)^{dr}\bigr)^{i/d},\qquad i,r\ge 1\;.
\end{equation}
\end{theorem}

\begin{proof}
Using Corollary~\ref{cor:Law} and the definition of the peak-descent polynomial in~\eqref{eqn:PeakDescent}, we obtain
\begin{align*}
\sum_{\pi\in S_n}\nu^{(p)}_{n,m}(\pi)X(\pi;\mathbf{x})
&=\frac{1}{m^n}[s^n](1+s)^m(1+ps)^{n-1}\sum_{\pi\in S_n}\beta(s)^{\des(\rho)}\alpha(s)^{\pk(\rho)}X(\pi\textcolor{blue}{;} {\bf x}) \\
&=\frac{1}{m^n}[s^n](1+s)^m(1+ps)^{n-1}\sum_{\rho\in S_n}\beta(s)^{\des(\rho)}\alpha(s)^{\pk(\rho)}X(\rho\textcolor{blue}{;} {\bf x}) \\
&=\frac{1}{m^n}[s^n](1+s)^m(1+ps)^{n-1}
\alpha(s)^{-1}\beta(s)^{-1}
F_n^{\pk,\des}(\alpha(s),\beta(s);\mathbf{x}).
\end{align*}
Setting $y=\frac{q}{p}$, and $t=\frac{1}{1+s}$, we have
\[
\frac{y+t}{1+yt}=\beta(s),
\qquad
\frac{(1+y)^2t}{(y+t)(1+yt)}=\alpha(s),
\qquad
\frac{1+yt}{1-t}=\frac{1+ps}{ps}.
\]
Substituting these in the {Gessel--Zhuang} formula~\eqref{eq:gz} and extracting the coefficient of $z^n$ gives
\begin{equation}
F_n^{\pk,\des}(\alpha(s),\beta(s);\mathbf{x})
=
\frac{p^n s^{n+1}}{(1+ps)^{n+1}}
[z^n]\mathcal{R}(s,z;\mathbf{x}),
\label{eq:F-from-R}
\end{equation}
where
\begin{align*}
\mathcal{R}(s,z;\mathbf{x})
:=
\sum_{k\ge 0}(1+s)^{-k}
\prod_{i\ge 1}
\exp\left\{
\sum_{r\ge 1}
\frac{(x_i z^i)^r}{ir}
\sum_{d\mid i}
\mu(d)
\left(k\left(1-\left(-\frac{q}{p}\right)^{dr}\right)\right)^{i/d}
\right\}.
\end{align*}
Since $\alpha(s)\beta(s)=\frac{1+s}{(1+ps)^2}$, using~\eqref{eq:F-from-R} we obtain
\begin{equation}
\sum_{\pi\in S_n}\nu^{(p)}_{n,m}(\pi)X(\pi;\mathbf{x})=\left(\frac{p}{m}\right)^n
[s^{-1}](1+s)^{m-1}[z^n]\mathcal{R}(s,z;\mathbf{x}).
\label{eq:residue-stage}
\end{equation}
We now use the elementary residue identity {for polynomials $P$}:
\begin{equation*}
[s^{-1}](1+s)^{m-1}
\sum_{k\ge 0}P(k)(1+s)^{-k}
=P(m)
\end{equation*}
Applying this identity in \eqref{eq:residue-stage}, we get $\sum_{\pi\in S_n}\nu^{(p)}_{n,m}(\pi)X(\pi;\mathbf{x})$ is equal to
\begin{align*}
&\left(\frac{p}{m}\right)^n[z^n]\prod_{i\ge 1}
\exp\left\{\sum_{r\ge 1}\frac{(x_i z^i)^r}{ir}
\sum_{d\mid i}\mu(d)
\left(m\left(1-\left(-\frac{q}{p}\right)^{dr}\right)\right)^{i/d}
\right\}.
\end{align*}
Thus, $\sum_{n\geq 0}u^n\sum_{\pi\in S_n}\nu_{n, m}^{(p)}(\pi)X(\pi;\x)$ is the same as replacing $z$ by $pu/m$ in the product on the right-hand side of the above identity.  The coefficient of \(x_i^r u^{ir}\) in the exponent then becomes
\[
\frac{1}{ir}
\left(\frac{p}{m}\right)^{ir}
\sum_{d\mid i}\mu(d)
\left(m\left(1-\left(-\frac{q}{p}\right)^{dr}\right)\right)^{i/d}
=
\frac{B^{(m,p)}_{i,r}}{ir},
\]
with \(B^{(m,p)}_{i,r}\) as in \eqref{eq:B-def}. This completes the proof.

\end{proof}

\begin{remark}
\label{remark:DFH-thm-35}
When $p=q=1/2$, we have
\[
p^{dr}-(-q)^{dr}=2^{-dr}(1-(-1)^{dr}).
\]
This is zero unless both \(d\) and \(r\) are odd.  When \(d\) and \(r\) both are odd, then
\[
\bigl(p^{dr}-(-q)^{dr}\bigr)^{i/d}
=2^{i/d-ir}.
\]
Therefore
\[
B^{(m,1/2)}_{i,r}=
\begin{cases}
\displaystyle
\sum_{\substack{d\mid i\\ d\text{ odd}}}
\mu(d)(2m)^{i/d-ir}, & r\text{ odd},\\[1.2em]
0,& r\text{ even}.
\end{cases}
\]
The logarithm of the right-hand side of \eqref{eq:asym-cycle-index} is then
\[
\sum_{i\ge 1}
\sum_{\substack{r\ge 1\\ r\text{ odd}}}
\frac{x_i^r u^{ir}}{ir}
\sum_{\substack{d\mid i\\ d\text{ odd}}}
\mu(d)(2m)^{i/d-ir}.
\]
Setting $f_{i,m}:=
\frac{1}{2i}
\sum_{\substack{d\mid i\\ d\text{ odd}}}
\mu(d)(2m)^{i/d}$, this is equal to 
\[
\sum_{i\ge 1}
 f_{i,m}
\sum_{\substack{r\ge 1\\ r\text{ odd}}}
\frac{2}{r}\left(x_i\left(\frac{u}{2m}\right)^i\right)^r
=
\sum_{i\ge 1}
f_{i,m}
\log\left(
\frac{1+x_i(u/2m)^i}{1-x_i(u/2m)^i}
\right)\;.
\]
Thus, we get

\begin{equation*}
1+
\sum_{n\ge 1}u^n
\sum_{\pi\in S_n}
\nu^{(1/2)}_{n,m}(\pi)X(\pi;\mathbf{x})
=
\prod_{i\ge 1}
\left(
\frac{1+x_i(u/2m)^i}{1-x_i(u/2m)^i}
\right)^{f_{i,m}}.
\label{eq:dfh-cycle-index}
\end{equation*}
This is precisely~\cite[Theorem 3.5]{DiaconisFulmanHolmes2013}.
\end{remark}

For $\x = (1, \ldots, 1)$, it is easy to check that $Z(u; \x)=(1-u)^{-1}$. Differentiating~\eqref{eq:asym-cycle-index} with respect to $x_i$ and then setting all $x_j=1$ immediately gives the generating function for the expected number of $i$-cycles; we record it as a corollary.
\begin{corollary}
\label{cor:expected-cycles}
For fixed $i\geq 1$, we have 
\[
\sum_{n\ge 0}u^n\E_{\nu^{(p)}_{n,m}}[N_i]=\frac{1}{1-u}\sum_{r\ge 1}\frac{B^{(m,p)}_{i,r}}{i}u^{ir}.
\]
In particular, 
\begin{equation*}
\E_{\nu^{(p)}_{n,m}}[N_i]
=
\frac{1}{i}
\sum_{r=1}^{\lfloor n/i\rfloor}
B^{(m,p)}_{i,r}.
\end{equation*}
\end{corollary}
\begin{corollary}[Expected number of fixed points]
\label{cor:fixedPoints}
The expected number of fixed points in $\pi\sim\nu_{n, m}^{(p)}$ is
\begin{equation*}
\E_{\nu^{(p)}_{n,m}}[N_1(\pi)]
=\sum_{r=1}^{n}m^{1-r}\bigl(p^r-(-q)^r\bigr)
=\frac{p(1-(p/m)^n)}{1-p/m}+\frac{q(1-(-q/m)^n)}{1+q/m}.
\end{equation*}
\end{corollary}

Theorem~\ref{thm:cycleIndex} has several other corollaries. We next state the following corollary, which generalizes~\cite [Corollary 3.6]{DiaconisFulmanHolmes2013}.

\begin{corollary}
\label{cor:geometric-and-limit}
Fix $m\geq 1$, $p\in (0, 1)$, and $u\in (0, 1)$.  
\begin{enumerate}
\item Let $N$ be geometrically distributed by
\[
\mathbb{P}(N=n)=(1-u)u^n,
\qquad n\ge 0\;,
\]
and, conditional on $N=n$, let $\pi$ have the law $\nu^{(p)}_{n,m}$. Then, every finite subcollection of $N_1,N_2,\ldots $ is independent and the probability generating function of $N_i$ is
\begin{equation}
G^{(m,p)}_{i,u}(x):=\E_{\pi}[x^{N_i}]=\exp\left\{\sum_{r\ge 1}\frac{x^r-1}{ir}B^{(m,p)}_{i,r}u^{ir}\right\}.
\label{eq:geom-pgf}
\end{equation}

\item Fix $i\geq 1$, and let $\pi_n\sim \nu^{(p)}_{n,m}$.  Then, for every finite set $I\subseteq\mathbb{N}$, the random variables $\{N_i(\pi_n):i\in I\}$ converge jointly to independent random variables with marginal pgfs $G^{(m,p)}_{i, 1}$. 
\end{enumerate}
\end{corollary}
\begin{proof}
    For any finite set $I$, set $x_i=1$ for all $i\notin I$ and observe that 
    \begin{align*}
        1+\sum_{n\geq 1}u^n\mathbb{E}_{\nu_{n, m}^{(p)}}\left[\prod_{i\in I}x_i^{N_i(\pi)}\right] &= \prod_{i\in I}\exp\left\{\sum_{r\geq 1}\frac{{B_{i, r}^{(m, p)}}}{ir}(x_i^r-1)u^{ir}\right\} Z(u, \1)\;.
    \end{align*}
    Recall that $Z(u, \1)=(1-u)^{-1}$, multiplying by $(1-u)$ on both sides, we get 
    \[
    \mathbb{E}_{\pi}\left[\prod_{i\in I}x_i^{N_i(\pi)}\right] = \prod_{i\in I}\exp\left\{\sum_{r\geq 1}\frac{{B_{i, r}^{(m, p)}}}{ir}(x_i^r-1)u^{ir}\right\}=: H_{I}(u, \x)\;,
    \]
    which proves the first part. 

    To prove the second part, we begin by noting that 
    \[
    1+\sum_{n\geq 1}u^n\mathbb{E}_{\nu_{n, m}^{(p)}}\left[\prod_{i\in I}x_i^{N_i(\pi)}\right] = \frac{H_{I}(u, \x)}{(1-u)}\;.
    \]
    We claim that $H_{I}(u, \x)$ is analytic in $u$ in a neighborhood of $1$ for all {${\bf x}$} satisfying $|x_i|\leq 1$. To see this, we note that for $i, r\geq 1$
    \[
    |B_{i, r}^{(m, p)}|\leq C_{i, m, p}\left(\frac{\max\{p, q\}}{m}\right)^{ir},
    \]
    for some constant $C_{i, m, p}>0$. Since $\max\{p, q\}<1$, it follows that 
    \[
    \sum_{r\geq 1}\frac{{B_{i, r}^{(m, p)}}}{ir}(x_i^r-1)u^{ir}
    \]
    is analytic in a ball of radius $>1$. Therefore, its exponential $H_{I}(u, \x)$ is also analytic in a ball of radius $>1$. Since $H_{I}(\cdot, {\bf x})$ is analytic in a small neighborhood of $1$, we get 
\[
    \lim_{n\to \infty} \mathbb{E}_{\nu_{n, m}^{(p)}}\left[\prod_{i\in I}x_i^{N_i(\pi)}\right] = \lim_{n\to \infty}[u^n] \frac{H_{I}(u, \x)}{(1-u)} = H_{I}(1, \x)\;.
\]
    This completes the proof.
\end{proof}

\begin{remark}
Let $p=q=1/2$ and let $f_{i,m}$ be as in Remark~\ref{remark:DFH-thm-35} and set $a_i(u):=\left(\frac{u}{2m}\right)^i$.
Under $\pi$ as in Corollary~\ref{cor:geometric-and-limit}\textup{(1)}, every finite subcollection of $(N_i)$ is independent, and
\begin{equation*}
\E_{\pi}[x^{N_i}]=
\left(\frac{1+x a_i(u)}{1+a_i(u)}\right)^{f_{i,m}}
\left(\frac{1-a_i(u)}{1-xa_i(u)}\right)^{f_{i,m}}.
\end{equation*}
Thus, $N_i$ is distributed as the convolution of
\[
\operatorname{Bin}\left(f_{i,m},\frac{a_i(u)}{1+a_i(u)}\right)
\qquad\text{and}\qquad
\operatorname{NegBin}\left(f_{i,m},a_i(u)\right),
\]
where \(\operatorname{NegBin}(f,a)\) has pgf \(((1-a)/(1-ax))^f\). For fixed \(m\) and \(i\), as \(n\to\infty\) under \(\nu^{(1/2)}_{n,m}\), every finite subcollection of cycle counts converges jointly to independent limits, and the limiting \(N_i\) is the convolution of
\[
\operatorname{Bin}\left(f_{i,m},\frac{1}{(2m)^i+1}\right)
\qquad\text{and}\qquad
\operatorname{NegBin}\left(f_{i,m},\left(\frac{1}{2m}\right)^i\right).
\]  
Thus, we recover~\cite[Corollary 3.6]{DiaconisFulmanHolmes2013}.
\end{remark}
The next corollary is the asymmetric analogue of~\cite[Corollary 3.7]{DiaconisFulmanHolmes2013}.  Let
\[L_1(\pi)\geq L_2(\pi)\geq \cdots\]
denote the cycle lengths of $\pi$, arranged in decreasing order.
\begin{corollary}[Large cycles]
\label{cor:large-cycles}
Fix $p\in(0,1)$ and $k\geq 1$. Let $m=m_n$ be any sequence of positive integers, fixed or growing with $n$.  If $\pi_n\sim \nu^{(p)}_{n,m_n}$ and $\sigma_n$ is uniform on $S_n$, then
\[
\begin{aligned}
\sup_{0\leq t_1,\ldots,t_k\leq 1}
\Bigg|&
\mathbb{P}\left(\frac{L_1(\pi_n)}{n}\leq t_1,\ldots,\frac{L_k(\pi_n)}{n}\leq t_k\right)\\
&-
\mathbb{P}\left(\frac{L_1(\sigma_n)}{n}\leq t_1,\ldots,\frac{L_k(\sigma_n)}{n}\leq t_k\right)
\Bigg|\longrightarrow 0.
\end{aligned}
\]
In particular, the ordered large cycles have the same $\operatorname{Poisson--Dirichlet}(1)$ limit as those of a uniform random permutation.
\end{corollary}

\begin{proof}
Let
\[
Z_{\infty}(u;\x)=\prod_{i\geq 1}\exp\left\{\frac{x_i u^i}{i}\right\}
\]
be the uniform cycle index.  By Theorem~\ref{thm:cycleIndex},
\[
Z_{m,p}(u;\x):=1+\sum_{n\geq1}u^n\sum_{\pi\in S_n}\nu^{(p)}_{n,m}(\pi)X(\pi;\x)
=Z_{\infty}(u;\x)A_{m,p}(u;\x),
\]
where
\[
A_{m,p}(u;\x)=\exp\left\{
\sum_{i,r\geq1}\frac{B^{(m,p)}_{i,r}-\1_{\{r=1\}}}{ir}x_i^r u^{ir}
\right\}.
\]
We first check that this extra factor has a radius of convergence larger than one, uniformly in $m$.  Since $p,q<1$, there are constants $C<\infty$ and $0<\gamma<1$, depending only on $p$, such that
\[
\left|B^{(m,p)}_{i,r}-\1_{\{r=1\}}\right|\leq C i\gamma^{ir},
\qquad i,r\geq1,
\]
uniformly in $m\geq1$.  Indeed, for $r=1$ the $d=1$ summand in $B^{(m,p)}_{i,1}$ is equal to $1$, and the remaining divisors $d\geq2$ are bounded by multiples of $(p^d+q^d)^{i/d}$.  For $r\geq2$, all divisors are bounded by multiples of $(p^{dr}+q^{dr})^{i/d}$.  In both cases the bases are bounded away from $1$, and the divisor factor is absorbed by increasing $\gamma$ slightly.  Hence $A_{m,p}$ is analytic for $|u|<1+\veps$ and $|x_i|\leq1$, with $\veps>0$ depending only on $p$.  Also $A_{m,p}(u;\1)=1$, because both $Z_{m,p}(u;\1)$ and $Z_{\infty}(u;\1)$ are equal to $(1-u)^{-1}$.

Write
\[
A_{m,p}(u;\x)=\sum_{\eta}a_{m,p}(\eta)u^{|\eta|}\prod_{i\geq1}x_i^{m_i(\eta)},
\]
where $\eta=(1^{m_1(\eta)}2^{m_2(\eta)}\cdots)$ ranges over integer partitions and $|\eta|=\sum_i i m_i(\eta)$.  The preceding absolute convergence implies that for some $R>1$,
\[
\sup_{m\geq1}\sum_{\eta}|a_{m,p}(\eta)|R^{|\eta|}<\infty,
\qquad
\sum_{\eta}a_{m,p}(\eta)=1.
\]
Let $\lambda(\rho)$ denote the partition of the integer $N$ formed by the cycle lengths of a permutation $\rho\in S_N$.  If $\lambda$ and $\eta$ are partitions, write $\lambda\cup\eta$ for the partition obtained by adjoining their parts and then arranging them in weakly decreasing order.  Coefficient extraction from $Z_{\infty}A_{m,p}$ gives, for every function $F$ on partitions of $n$ with $|F|\leq1$,
\[
\E_{\nu^{(p)}_{n,m}}F(\lambda(\pi))=
\sum_{|\eta|\leq n}a_{m,p}(\eta)
\E\left[F\left(\lambda(\sigma_{n-|\eta|})\cup\eta\right)\right],
\]
where $\sigma_N$ is uniform on $S_N$.

Now take
\[
F(\lambda)=\1\left\{L_1(\lambda)/n\leq t_1,\ldots,L_k(\lambda)/n\leq t_k\right\}.
\]
For each fixed $\eta$, adjoining the parts of $\eta$ changes the ranked cycle lengths by $o(n)$ after normalization by $n$.  Thus the classical convergence of uniform permutation cycles to Poisson--Dirichlet$(1)$, whose finite-dimensional distribution functions are continuous, gives
\[
\sup_{0\leq t_1,\ldots,t_k\leq1}
\left|
\E F\left(\lambda(\sigma_{n-|\eta|})\cup\eta\right)-
\E F\left(\lambda(\sigma_n)\right)
\right|\longrightarrow0
\]
for every fixed $\eta$.  The convergence is uniform over the finitely many $\eta$ with $|\eta|\leq M$, and the exponentially small bound above contributes $|\eta|>M$ arbitrarily small, uniformly in $m$.  This proves the asserted uniform convergence. 
\end{proof}

Finally, we record the following generalization of~\cite[Corollary 3.8]{DiaconisFulmanHolmes2013}.
\begin{corollary}[Joint distribution of descents, valleys, and cycle type]
\label{cor:joint-des-val-cycle}
For $t\in (0, 1)$, as a formal power series in $u$, one has
\begin{align}
&\frac{t}{1-t}+\sum_{n\ge1}u^n\sum_{\pi\in S_n}\frac{t^{\vl(\pi)+1}}{p^n(1-t)^{n+1}}(p+qt)^{\asc(\pi)-\vl(\pi)}
(q+pt)^{\des(\pi)-\vl(\pi)}X(\pi;\x)
\label{eq:joint-des-val-cycle}
\\
&\qquad=\sum_{m\ge1}t^m\prod_{i\ge1}
\exp\left\{
\sum_{r\ge1}
\frac{x_i^r u^{ir}}{ir}
\sum_{d\mid i}
\mu(d)m^{i/d}
\left(1-\left(-\frac{q}{p}\right)^{dr}\right)^{i/d}
\right\}.
\notag
\end{align}
The same identity holds with \(\vl(\pi)\) replaced by \(\pk(\pi)\).
\end{corollary}
\begin{proof}
Recall from Theorem~\ref{thm:Law} that
\[
        \nu^{(p)}_{n,m}(\pi)
        =\frac{1}{m^n}[s^n](1+s)^m
        (1+ps)^{\asc(\pi)-\vl(\pi)}
        (1+qs)^{\des(\pi)-\vl(\pi)}
        (1+s)^{\vl(\pi)}.
\]
Multiplying by $t^m(\frac{m}{p})^n$ and summing over $m\geq 1$, we obtain $\sum_{m\ge1}t^m
\left(\frac{m}{p}\right)^n\nu^{(p)}_{n,m}(\pi)$ equals 
\begin{align*}
&=p^{-n}[s^n]\frac{t(1+s)}{1-t(1+s)}(1+ps)^{\asc(\pi)-\vl(\pi)}(1+qs)^{\des(\pi)-\vl(\pi)}(1+s)^{\vl(\pi)}.
\end{align*}
A simple residue calculation shows that for every polynomial $H$ of degree at most $n-1$, we have
\[
[s^n]\frac{t(1+s)}{1-t(1+s)}H(s)=\frac{t^n}{(1-t)^{n+1}}H\left(\frac{1-t}{t}\right).
\]
Applying this with $H(s)=(1+ps)^{\asc(\pi)-\vl(\pi)}(1+qs)^{\des(\pi)-\vl(\pi)}(1+s)^{\vl(\pi)}$, we get
\[
\sum_{m\ge1}t^m\left(\frac{m}{p}\right)^n\nu^{(p)}_{n,m}(\pi)= \frac{t^{\vl(\pi)+1}}{p^n(1-t)^{n+1}}(p+qt)^{\asc(\pi)-\vl(\pi)}(q+pt)^{\des(\pi)-\vl(\pi)}.
\]
Multiplying by $u^nX(\pi;\x)$ and summing over $n\geq 0$ and $\pi\in S_n$, we conclude the proof.

Since conjugation does not change the cycle type and conjugation by $w_0=(n, n-1, \ldots, 1)$ preserves the $\des$ and $\asc$ while changing the $\vl$ to $\pk$, the same conclusion holds with $\vl$ replaced by $\pk$. 

\end{proof}

\subsection{The number of descents}
\label{sec:Descents}
We now derive the generating function of the number of descents in $\pi\sim \nu_{n, m}^{(p)}$, and use it to obtain the mean and variance of the number of descents. Let us reiterate that this is different from the discussion in~\cite[Section 3.6]{DiaconisFulmanHolmes2013} and the results in this section are new even for $p=1/2$.

To simplify some expressions, we introduce some notations. Given $u, t, p$ and  $q=(1-p)$, we set 
\[
A_t(u):= 1+t+(p+qt)u\;.
\]
Let $\omega_t(u)$ be the unique solution to 
\[
\frac{\omega_t(u)}{(1+\omega_t(u))^2} = \frac{t(1+u)}{A_t(u)^2}, \qquad \omega_1(u) = \frac{1}{1+u}\;.
\]


\begin{theorem}
\label{thm:GF-Descents}
Let $m, n\geq 1$ and $p\in (0, 1)$ be fixed. Let $G_{n, m}^{(p)}(t) := \mathbb{E}_{\nu_{n, m}^{(p)}}[t^{\des(\pi)+1}]$.
Then, 
\[
G_{n, m}^{(p)}(t) = \frac{1}{m^n}[u^n](1+u)^{m-1}\left(A_t(u)\frac{1-\omega_t(u)}{1+\omega_t(u)}\right)^{n+1}\,\, \sum_{k\geq 1}k^n\omega_t(u)^{k}\;.
\]
\end{theorem}
\begin{proof}
Recall that we define 
\[
F_n^{\pk,\des}(y,t)= \sum_{\rho\in S_n}
y^{\pk(\rho)+1}t^{\des(\rho)+1}\;.
\]
Let $\alpha(u)=\frac{1+u}{(1+pu)(1+qu)}$ and $\beta(u) = \frac{1+qu}{1+pu}$. Using Corollary~\ref{cor:Law} and the definition of $F_n^{\pk, \des}$, we get
\begin{align}
G_{n, m}^{(p)}(t) &= \frac{1}{m^n}[u^n](1+u)^{m}(1+pu)^{n-1}\frac{F_n^{\pk,\des}(\alpha(u),t\beta(u)\textcolor{blue}{)}}{\alpha(u)\beta(u)}\nonumber\\
&= \frac{1}{m^n}[u^n](1+u)^{m-1}(1+pu)^{n+1}F_n^{\pk,\des}(\alpha(u),t\beta(u)\textcolor{blue}{)}\;,
\label{eqn:G-definition}
\end{align}
where we used that $\alpha(u)\beta(u)= (1+u)/(1+pu)^2$ in the last equality. Set 
\[r=\omega_t(u), \qquad y=\frac{t\beta(u)-r}{1-rt\beta(u)}.\] It can be verified by a direct computation that
\[
\frac{(1+y)^2r}{(y+r)(1+yr)}= \alpha(u), \qquad \frac{y+r}{1+yr} = t\beta(u), \qquad \frac{(1+y)(1-r)}{(1+yr)} = (1+t\beta(u))\frac{1-r}{1+r}\;.
\]
Using~\eqref{eq:GZ-run}, we obtain 
\[
F_n^{\pk, \des}(\alpha(u), t\beta(u)) = \left[(1+t\beta(u))\frac{1-r}{1+r}\right]^{n+1}\,\sum_{k\geq 1}k^nr^k\;.
\]
The proof follows by combining this with $(1+pu)(1+t\beta(u)) = A_{t}(u)$ and~\eqref{eqn:G-definition}.

\end{proof}

We now use Theorem~\ref{thm:GF-Descents} to find the mean and variance of the number of descents in the shelf shuffle. It will be convenient to define the Faulhaber polynomial
\[S_j(m):=\sum_{a=0}^{m-1}a^j\;\qquad j, m\geq 1\;.\]
We refer an interested reader to~\cite{Knuth, Kellner} for an interesting account of the history of these polynomials as well as their properties.

\begin{corollary}[First two descent moments]
\label{cor:MomentsOfDescents}
Let $n,m\geq 1$, let $\pi\sim\nu_{n,m}^{(p)}$, and let $D_n=\des(\pi)$. Put $\dd=2p-1$. Then
\begin{equation}
\label{eq:mean-forward}
\E[D_n]
=
\frac{n-1}{2}
+
\dd\left(
 m-\frac{n+1}{2}
 -\frac{n+1}{m^n}S_n(m)
\right).
\end{equation}
Moreover,
\begin{align}
\label{eq:second-factorial-forward}
\E[D_n(D_n+1)]
&=
\frac{n(n+1)}{4}-\frac{m}{2}
+\frac{(n+1)mS_n(m)-(n+2)S_{n+1}(m)}{m^n} \notag\\
&
+\dd\left(
 nm-\frac{n(n+1)}{2}
 -\frac{n(n+1)}{m^n}S_n(m)
\right) \notag\\
&
+\dd^2\left(
 m^2-\left(n+\frac12\right)m+\frac{n(n+1)}{4}\right. \notag\\
&\qquad\left.
 +\frac{(n^2-1)mS_n(m)-n(n+2)S_{n+1}(m)}{m^n}
\right).
\end{align}
\end{corollary}

\begin{proof}
Let $q=1-p$, and, for fixed $k$, set
\[
H_k(t,u):=
\left(A_t(u)\frac{1-\omega_t(u)}{1+\omega_t(u)}\right)^{n+1}\omega_t(u)^k.
\]
At $t=1$,
\[
A_1(u)=2+u,
\qquad
\omega_1(u)=\frac{1}{1+u},
\qquad
H_k(1,u)=u^{n+1}(1+u)^{-k}.
\]
Logarithmic differentiation of
\[
\frac{\omega_t(u)}{(1+\omega_t(u))^2}
=
\frac{t(1+u)}{A_t(u)^2}
\]
gives
\[
\left.\frac{\partial}{\partial t}\log \omega_t(u)\right|_{t=1}=\dd,
\qquad
\left.\frac{\partial}{\partial t}\log\left(A_t(u)\frac{1-\omega_t(u)}{1+\omega_t(u)}\right)\right|_{t=1}
=q-\frac{\dd}{u}.
\]
Hence
\[
\left.\frac{\partial}{\partial t}H_k(t,u)\right|_{t=1}
=u^{n+1}(1+u)^{-k}
\left((n+1)q+\dd k-\frac{(n+1)\dd}{u}\right).
\]

We shall use the following elementary coefficient identities. If $P$ is a polynomial with $P(0)=0$, then
\begin{align*}
[u^{-1}](1+u)^{m-1}\sum_{k\geq1}P(k)(1+u)^{-k}&=P(m),\\
[u^{0}](1+u)^{m-1}\sum_{k\geq1}P(k)(1+u)^{-k}&=\sum_{a=0}^{m-1}P(a),\\
[u^{1}](1+u)^{m-1}\sum_{k\geq1}P(k)(1+u)^{-k}&=\sum_{a=0}^{m-1}(m-1-a)P(a).
\end{align*}

Since $G_{n,m}^{(p)}(t)=\E[t^{D_n+1}]$, Theorem~\ref{thm:GF-Descents} and the first two coefficient identities, applied to the formula above for $\partial_tH_k$, give
\[
\frac{d}{dt}G_{n, m}^{(p)}(t)\big\vert_{t=1} = \E[D_n+1]
=(n+1)q+\dd m-\frac{(n+1)\dd}{m^n}\sum_{a=0}^{m-1}a^n.
\]
Subtracting $1$ and using $q=(1-\dd)/2$ proves \eqref{eq:mean-forward}.

For the second derivative, differentiating once more and simplifying gives
\[
\left.\frac{\partial^2}{\partial t^2}H_k(t,u)\right|_{t=1}
=
 u^{n+1}(1+u)^{-k}\bigl(Q_0(k,u)+\dd Q_1(k,u)+\dd^2Q_2(k,u)\bigr),
\]
where
\begin{align*}
Q_0(k,u)
&=-\frac{k}{2}+\frac{n(n+1)}{4}
+\frac{-k+n+1}{u}
+\frac{n+1}{u^2},\\
Q_1(k,u)
&=nk-\frac{n(n+1)}{2}-\frac{n(n+1)}{u},\\
Q_2(k,u)
&=k^2-\left(n+\frac12\right)k+\frac{n(n+1)}{4}
+\frac{-(2n+1)k+n^2-1}{u}
+\frac{n^2-1}{u^2}.
\end{align*}
Applying the three coefficient identities to the $u^0,u^{-1},u^{-2}$ parts of $Q_0,Q_1,Q_2$ gives the three displayed contributions in \eqref{eq:second-factorial-forward}. 

\end{proof}

When $p=q=1/2$, note that $\delta=0$. This simplifies the expressions for the mean and variance of the descents. We record it in the following corollary.
\begin{corollary}[Descents in symmetric shelf shuffle]
When $p=q=1/2$, we have 
\begin{align*}
\mathbb{E}[D_n] &= \frac{n-1}{2}, \\
\Var(D_n) &=\frac{n+1}{4}-\frac{m}{2}
+\frac{(n+1)mS_n(m)-(n+2)S_{n+1}(m)}{m^n}.
\end{align*}
\end{corollary}

Using the fact that $S_{n}(m)\sim \frac{m^{n+1}}{(n+1)}$, it is easy to see that when $p=1/2$, the variance of the descents satisfies 
\[
\frac{\Var(D_n)}{n} = \frac{1}{4} + O(m^2/n)\;.
\]
In particular, ${\Var(D_n)}/{n}\to 1/4$ if $m^2/n\to 0$. This observation was first made by Alexander Clay. On the other hand, let $\nu_{n, \infty}$ denote the uniform measure on $S_n$. It is well-known that 
\[\mathbb{E}_{\nu_{n, \infty}}[\des(\pi)]=\frac{n-1}{2}, \quad \text{and} \quad\Var_{\nu_{n, \infty}}[\des(\pi)] = {\frac{n+1}{12}}.\] 
When $m\approx n^{3/2}$, we know that $\nu_{n, m}^{(p)}$ is close to the uniform distribution $\nu_{n, \infty}$ for every $p\in (0, 1)$. One might, therefore, expect that if $m/n^{3/2}\to \infty$, then $\Var_{\nu_{n, m}^{(p)}}[D_n]/n \to 1/12$. This is indeed true, and it was also first noted by Clay. In the following theorem, we strengthen these observations by proving the asymptotics for $\Var(D_n)/n$ in the regime $m=o(n)$ and $n=o(m)$.

\begin{theorem}[Asymptotics of $\Var(D_n)$]
\label{thm:varianceAsymptotics}
Let $D_n$ be the number of descents of an $(m, p)$-shelf shuffle with $n$ cards. Put $q=1-p$. Then the following limits hold, uniformly in the choice of $p=p(n)$:
\begin{enumerate}
\item If $m/n\to0$, then
\[
\frac{\operatorname{Var}(D_n)}{n}-pq\longrightarrow0.
\]
In particular, for fixed $p$, $\operatorname{Var}(D_n)/n\to pq$.
\item If $m/n\to\infty$, then
\[
\frac{\operatorname{Var}(D_n)}{n}\longrightarrow \frac1{12}.
\]
\end{enumerate}
\end{theorem}

\begin{proof}
Put $\delta=2p-1$. For $1\leq \ell\leq m$, define
\[
t_\ell=\left(1-\frac{\ell}{m}\right)^n,
\qquad
A_{n,m}=\sum_{\ell=1}^m t_\ell,
\qquad
B_{n,m}=\sum_{\ell=1}^m \ell t_\ell\textcolor{blue}{.}
\]
The exact variance formula may be written in the form
\begin{equation}\label{eq:variance-decomposition-snippet}
\operatorname{Var}(D_n)=V_{n,m}^{(0)}+\delta^2 W_{n,m},
\end{equation}
where
\begin{align*}
V_{n,m}^{(0)}
&=
\frac{n+1-2m}{4}-mA_{n,m}+(n+2)B_{n,m},\\
W_{n,m}
&=
-\frac{n+1-2m}{4}+mA_{n,m}-B_{n,m}
+(n+1)^2\bigl(B_{n,m}-A_{n,m}-A_{n,m}^2\bigr).
\end{align*}
Indeed, this is obtained from the exact formula for $\operatorname{Var}(D_n)$ by the change of variables $a=m-\ell$ and collecting the terms that are independent of $\delta^2$ and proportional to $\delta^2$.

First assume $m/n\to0$, and set $\alpha=n/m$. Since $t_\ell\leq e^{-\alpha \ell}$, we have
\[
0\leq A_{n,m}\leq \frac{e^{-\alpha}}{1-e^{-\alpha}}\to0,
\qquad
0\leq B_{n,m}\leq \frac{e^{-\alpha}}{(1-e^{-\alpha})^2}\to0.
\]
It remains only to control the term $B_{n,m}-A_{n,m}-A_{n,m}^2$, because it is multiplied by $(n+1)^2$ in \eqref{eq:variance-decomposition-snippet}. Observe that
\[
B_{n,m}-A_{n,m}=\sum_{\ell=1}^m(\ell-1)t_\ell
=\sum_{\substack{a,b\geq1\\ a+b\leq m}}t_{a+b}.
\]
For $a+b\leq m$, one has $t_{a+b}\leq t_at_b$. Hence, if
\[
C_{n,m}:=B_{n,m}-A_{n,m}-A_{n,m}^2,
\]
then $C_{n,m}\leq0$. Moreover, by the mean-value theorem, for $a+b\leq m$,
\begin{align*}
0\leq t_at_b-t_{a+b}
&=\left(1-\frac{a}{m}\right)^n\left(1-\frac{b}{m}\right)^n
-\left(1-\frac{a+b}{m}\right)^n\\
&\leq \frac{nab}{m^2}\exp\left(-\frac{(n-1)(a+b)}{m}\right).
\end{align*}
The pairs with $a+b>m$ contribute at most $m^2e^{-n}$. Therefore, with $\beta=(n-1)/m$,
\begin{align*}
0\leq -C_{n,m}
&\leq
\frac{n}{m^2}\left(\sum_{a\geq1}ae^{-\beta a}\right)^2+m^2e^{-n}.
\end{align*}
Multiplying by $n$ gives
\[
0\leq n(-C_{n,m})
\leq
\left(\frac{n}{m}\right)^2
\left(\frac{e^{-\beta}}{(1-e^{-\beta})^2}\right)^2
+nm^2e^{-n}\longrightarrow0,
\]
because $\beta\sim n/m\to\infty$ and $m=o(n)$. Thus $nC_{n,m}\to0$. It follows that
\[
\frac{V_{n,m}^{(0)}}{n}\to \frac14,
\qquad
\frac{W_{n,m}}{n}\to -\frac14.
\]
Using \eqref{eq:variance-decomposition-snippet},
\[
\frac{\operatorname{Var}(D_n)}{n}
\to \frac14-\frac{\delta^2}{4}
=\frac{1-(2p-1)^2}{4}=pq,
\]
or, if $p=p(n)$ varies, equivalently
$\operatorname{Var}(D_n)/n-pq\to0$.

Now assume $m/n\to\infty$. Let
\[
S_r(m)=\sum_{j=0}^{m-1}j^r.
\]
Then
\[
A_{n,m}=\frac{S_n(m)}{m^n},
\qquad
B_{n,m}=\frac{mS_n(m)-S_{n+1}(m)}{m^n}.
\]
The Euler--Maclaurin expansion for power sums, uniformly in the range $r/m\to0$, gives
\[
S_r(m)=\frac{m^{r+1}}{r+1}-\frac{m^r}{2}+\frac{r m^{r-1}}{12}
-\frac{r(r-1)(r-2)m^{r-3}}{720}
+O\left(r^5m^{r-5}\right).
\]
Substituting this expansion with $r=n$ and $r=n+1$ in the preceding expressions for $A_{n,m}$ and $B_{n,m}$, and then simplifying in \eqref{eq:variance-decomposition-snippet}, yields
\begin{align*}
V_{n,m}^{(0)}
&=\frac{n+1}{12}+\frac{n(n^2-1)}{180m^2}
+O\left(\frac{n^5}{m^4}\right),\\
W_{n,m}
&=-\frac{n(n+1)(6n-1)}{360m^2}
+O\left(\frac{n^5}{m^4}\right).
\end{align*}
Since $|\delta|\leq1$, \eqref{eq:variance-decomposition-snippet} implies
\[
\frac{\operatorname{Var}(D_n)}{n}
=\frac{n+1}{12n}
+O\left(\frac{n^2}{m^2}+\frac{n^4}{m^4}\right).
\]
The error tends to zero because $m/n\to\infty$, and hence
\[
\frac{\operatorname{Var}(D_n)}{n}\to\frac1{12}.
\]
\end{proof}

\subsection{Number of inversions}
\label{sec:Inversions}

In this section, we study the number of inversions in the shelf shuffle. For a permutation $\pi\in S_n$, we define
\[
\operatorname{Inv}(\pi):=\{(i,j):1\leq i<j\leq n,\ \pi_i>\pi_j\}, \qquad \operatorname{inv}(\pi) = |\operatorname{Inv}(\pi)|\;.
\]
It is easy to see that for any permutation $\pi\in S_n$, we have $\inv(\pi)=\inv(\pi^{-1})$. We first record a useful inverse description of the asymmetric shelf shuffle. We then obtain the generating function identity for $\inv$ and compute the mean and the variance for the number of inversions in $\pi\sim\nu_{n, m}^{(p)}$.

In this section and the following sections, it will be more convenient to work with an alternate description of the shelf-shuffle that we describe below. For $m\geq 1$, let
\[
\mathcal A_m:=\{1^+,1^-,2^+,2^-,\ldots,m^+,m^-\},\]
and we equip $\mathcal{A}_m$ with the following total order
\[
1^+<1^-<2^+<2^-<\cdots<m^+<m^-.
\]
For every $1\leq j\leq m$, we assign the weight $\wt(j^+)=p/m$ and $\wt(j^-)=q/m$.  For a word $a=(a_1,\ldots,a_n)\in \mathcal A_m^n$, define $\operatorname{std}_{\pm}(a)\in S_n$ by ranking the positions $1,\ldots,n$ in the order of their letters, breaking ties among equal $j^+$ letters in increasing order of positions and ties among equal $j^-$ letters in decreasing order of positions.

\begin{lemma}
\label{lem:mixed-standardization}
Let $a_1,\ldots,a_n$ be independent letters of $\mathcal A_m$ with the weight $\wt$ as above. Let $\pi\sim \nu^{(p)}_{n,m}$.  Then
\[
\pi^{-1}\stackrel{d}{=}\operatorname{std}_{\pm}(a).
\]
Consequently, $\operatorname{inv}(\pi)=\operatorname{inv}(\operatorname{std}_{\pm}(a))$ in distribution.
\end{lemma}

\begin{proof}
Encode a card by $j^+$ if it is assigned shelf $j$ and placed on top of that shelf, and by $j^-$ if it is assigned shelf $j$ and placed on the bottom.  The shelves are assembled in increasing order.  Within a fixed shelf, all top-placed cards appear above all bottom-placed cards; since cards are dealt from the bottom of the original deck, the top-placed cards appear in increasing order of their labels, while the bottom-placed cards appear in decreasing order of their labels.  Thus, the final position of card $i$ is exactly the rank assigned to position $i$ by the mixed standardization rule.  This proves the first assertion.  The second follows because the inversion number is invariant under taking inverses.
\end{proof}

For $r\geq0$, we write $(a;t)_r:=\prod_{j=0}^{r-1}(1-at^j)$ and $(a;t)_\infty:=\prod_{j\geq0}(1-at^j)$, as a formal power series.  We also write
\[
\left[\begin{matrix} n\\ c_1,\ldots,c_{\ell}\end{matrix}\right]_t:=\frac{(t;t)_n}{(t;t)_{c_1}\cdots (t;t)_{c_{\ell}}} = \frac{\prod_{j=1}^{n}(1-t^j)}{\prod_{k=1}^{\ell}\prod_{j=1}^{c_{k}}(1-t^{j})}\;
\]
for the $t$-multinomial coefficient.

Finally, we define the set 
\[
\mathcal{AB}(m, n) :=\left\{(a_1, b_1, \ldots, a_m, b_m): a_i\geq 0, b_i\geq 0 \text{ for all } i, \sum_{i=1}^{m}(a_i+b_i)=n\right\}\;.
\]
We will denote the elements in $\mathcal{AB}(m, n)$ by $({\bf a}, {\bf b})$ and we will write $|{\bf a}|=\sum_{i=1}^{m}a_i$ and $|{\bf b}|=\sum_{i=1}^{m}b_i$.  We also define
\[
\left[\begin{matrix} n\\ {\bf (a, b)}\end{matrix}\right]_t := \left[\begin{matrix} n\\ a_1, b_1, \ldots,a_{m}, b_m\end{matrix}\right]_t\;.
\]

\begin{theorem}[Generating function for inversions]
\label{thm:inv-gf}
Let $\pi\sim\nu^{(p)}_{n,m}$.  Then
\begin{align}
\E\left[t^{\operatorname{inv}(\pi)}\right]
&=
\sum_{{\bf (a, b)}\in \mathcal{AB}(m, n)}
\left[\begin{matrix} n\\ {\bf (a, b)}\end{matrix}\right]_t
\cdot \prod_{j=1}^{m}t^{\binom{b_j}{2}}\cdot 
\left(\frac{p}{m}\right)^{|{\bf a}|}\left(\frac{q}{m}\right)^{|{\bf b}|}
\label{eq:inv-gf-finite}\\
&=\frac{(t;t)_n}{m^n}[z^n]
\left(\frac{(-qz;t)_\infty}{(pz;t)_\infty}\right)^m.
\label{eq:inv-gf-product}
\end{align}
\end{theorem}

\begin{proof}
We will work with the labelled word $a$ from Lemma~\ref{lem:mixed-standardization}.  Suppose that a realization of $a$ is given by $a_j$ copies of $j^+$ and $b_j$ copies of $j^-$, $1\leq j\leq m$.  The inversions in $\operatorname{std}_{\pm}(a)$ coming from unequal labels are precisely the inversions of the word $a\in \mathcal A_m^n$.  Equal $j^+$ labels contribute no inversions, while equal $j^-$ labels contribute $\binom{b_j}{2}$ inversions. The inversion enumerator of words with this fixed realization is, therefore, the $t$-multinomial coefficient, giving \eqref{eq:inv-gf-finite}.  Finally, \eqref{eq:inv-gf-product} follows from the standard $t$-binomial identities (see~\cite[Eqs. (1.3.15), (1.3.16)]{Gasper04})
\[
\sum_{a\geq0}\frac{x^a}{(t;t)_a}=\frac{1}{(x;t)_\infty},
\qquad
\sum_{b\geq0}\frac{t^{\binom{b}{2}}x^b}{(t;t)_b}=(-x;t)_\infty.
\]
\end{proof}

\begin{corollary}[Mean and variance for inversions]
\label{cor:inv-moments}
Fix $n, m \geq 1$ and $p\in (0, 1)$.
Let $\pi\sim\nu^{(p)}_{n,m}$ and set $I_n=\operatorname{inv}(\pi)$ and $\dd=2p-1$.  Then,
\[
\E[I_n]=\binom{n}{2}\left(\frac12-\frac{\dd}{2m}\right),
\]
and
\[
\Var(I_n)=\frac{n(n-1)(2n+5)}{72}
+\frac{n(n-1)(n-2)}{18m^2}
-\frac{\dd^2 n(n-1)(2n-1)}{24m^2}.
\]
\end{corollary}

\begin{proof}
Let $F_n(t):=\E[t^{I_n}]$ be the generating function in Theorem~\ref{thm:inv-gf}. Define $\mathcal D=t\frac{d}{dt}$.  Since $\mathcal D t^r=rt^r$, we have
\[
\mathcal D F_n(1)=\E[I_n],
\qquad
\mathcal D^2F_n(1)=\E[I_n^2].
\]
We will differentiate~\eqref{eq:inv-gf-finite}.  Fix a ${\bf (a, b)}\in \mathcal{AB}(m, n)$ and write $K({\bf b}):=\sum_{j=1}^m\binom{b_j}{2}$.
For ${\bf (a, b)}$, write
\[
R_{\bf (a, b)}(t) := 
\left[\begin{matrix} n\\ a_1,b_1,\ldots,a_m,b_m\end{matrix}\right]_1^{-1}
\left[\begin{matrix} n\\ a_1,b_1,\ldots,a_m,b_m\end{matrix}\right]_t
t^{K({\bf b})}.
\]
Note that $R_{\bf (a, b)}(1)=1$. Let $C=(A_1,B_1,\ldots,A_m,B_m)$ bedrawn from the multinomial distribution with parameters $(p/m, q/m, \ldots, p/m, q/m)$. Then,
\[
\E[I_n]= \E\left[\left.\mathcal D\log R_C(t)\right|_{t=1}\right],
\]
and
\[
\Var(I_n)=\E\left[\left.\mathcal D^2\log R_C(t)\right|_{t=1}\right]+
\Var\left(\left.\mathcal D\log R_C(t)\right|_{t=1}\right).
\]
We will now compute {$\mathcal D\log R_{\bf (a, b)}(t)$} and {$\mathcal D^2\log R_{\bf (a, b)}(t)$}. To this end, for $r\geq0$, we write
\[
[r]_t!:=\prod_{s=1}^{r}(1+t+\cdots+t^{s-1}).
\]
Since, after division by $s$, the polynomial $1+t+\cdots+t^{s-1}$ is the probability generating function of the uniform distribution on $\{0,\ldots,s-1\}$, we get
\[
\left.\mathcal D\log [r]_t!\right|_{t=1}
=
\frac{r(r-1)}{4},
\qquad
\left.\mathcal D^2\log [r]_t!\right|_{t=1}
=
\frac{r(r-1)(2r+5)}{72}.
\]
Using this and the fact that
\[
\left[\begin{matrix} n\\ a_1,b_1,\ldots,a_m,b_m\end{matrix}\right]_t
=
\frac{[n]_t!}{[a_1]_t![b_1]_t!\cdots [a_m]_t![b_m]_t!},
\]
we conclude that
\[
\left.\mathcal D\log R_{\bf (a, b)}(t)\right|_{t=1}
=
{K({\bf b})}+\frac14\left(n^2-\sum_{j=1}^m(a_j^2+b_j^2)\right) = \frac{n(n-1)}{4}+\frac{U({\bf (a, b)})}{4},
\]
where
\[
U({\bf (a, b)}):=\sum_{j=1}^m\left(b_j(b_j-1)-a_j(a_j-1)\right)\;.
\]
Similarly, we also get
\begin{align*}
\left.\mathcal D^2\log R_{\bf (a, b)}(t)\right|_{t=1}
&=\frac{1}{72}\bigg(n(n-1)(2n+5)
-\sum_{j=1}^m a_j(a_j-1)(2a_j+5)\\
&\hspace{3.4cm}
-\sum_{j=1}^m b_j(b_j-1)(2b_j+5)\bigg).
\end{align*}
We now use the standard multinomial factorial-moment identity
\[
\E\left[\prod_i C_i^{\underline{r_i}}\right]
=
n^{\underline{r_1+r_2+\cdots}}
\prod_i \rho_i^{r_i},
\]
where the $C_i$ are multinomial cell counts with cell probabilities $\rho_i$, and
$x^{\underline r}=x(x-1)\cdots(x-r+1)$.  First,
\[
\E[U(C)]
=
n(n-1)\left(m\frac{q^2}{m^2}-m\frac{p^2}{m^2}\right)
=
-\frac{\dd n(n-1)}{m}.
\]
Hence
\[
\E[I_n]
=
\frac{n(n-1)}4-\frac{\dd n(n-1)}{4m}
=
\binom{n}{2}\left(\frac12-\frac{\dd}{2m}\right).
\]
For the variance, we use
\[
x(x-1)(2x+5)=2x^{\underline3}+9x^{\underline2},
\qquad
\left(x^{\underline2}\right)^2=x^{\underline4}+4x^{\underline3}+2x^{\underline2}
\]
to write
\[
\E\left[\left.\mathcal D^2\log R_C(t)\right|_{t=1}\right]
=
\frac{n(n-1)(2n+5)}{72}
-\frac{n(n-1)(p^2+q^2)}{8m}
-\frac{n(n-1)(n-2)(p^3+q^3)}{36m^2},
\]
and
\[
\Var(U(C))
=
\frac{2n(n-1)(p^2+q^2)}{m}
+
\frac{4n(n-1)(n-2)(p^3+q^3)}{m^2}
-
\frac{2n(n-1)(2n-3)\dd^2}{m^2}.
\]
Therefore
\begin{align*}
\Var(I_n)
&=
\frac{n(n-1)(2n+5)}{72}
+
\frac{2n(n-1)(n-2)(p^3+q^3)}{9m^2}
-
\frac{n(n-1)(2n-3)\dd^2}{8m^2}.
\end{align*}
We get the desired formula by noting that $p^3+q^3={(1+3\dd^2)}/{4}$.
\end{proof}

\subsection{Distribution of RSK shape and inverse descents}
\label{sec:RSK-and-inv-descent}
We now study the distribution of {Robinson--Schensted--Knuth} (RSK) shape and the distribution of descents in $\pi^{-1}$, where $\pi\sim \nu_{n, m}^{(p)}$. For $p=1/2$, this was carried out in~\cite[Section 3.5, Section 3.6]{DiaconisFulmanHolmes2013}.

We recall that the RSK correspondence gives a bijection between $S_n$ and the set of pairs of $(P, Q)$ of standard Young tableaux of the same shape and size $n$. For the details of this bijection and more generally, for the background, we refer the reader to~\cite{Stan24}. For a permutation $\sigma$, write $\operatorname{sh}(\sigma)$ for its RSK shape and write $\operatorname{RSK}(\sigma)=(P(\sigma),Q(\sigma))$. The tableau $P(\sigma)$ is called the insertion tableau, and $Q(\sigma)$ is called the recording tableau of the permutation $\sigma$.

For a partition $\lambda$, let $\operatorname{SSYT}^{\pm}_{m}(\lambda)$ be the set of fillings of the Young diagram of $\lambda$ with letters in $\mathcal A_m$ such that entries are weakly increasing in rows and columns, each letter $j^+$ appears at most once in each column, and each letter $j^-$ appears at most once in each row. We also make a standard abuse of notation and write $b\in \lambda$ to denote a box/cell in the Young diagram with shape $\lambda$. For $T\in \operatorname{SSYT}^{\pm}_{m}(\lambda)$ and $b\in \lambda$, we use $T(b)$ to denote the symbol in $\mathcal{A}_m$ that appears in the box $b$.

With this preparation, we are ready to state our next result that describes the distribution of the shape $\operatorname{sh}(\pi)$ when $\pi\in \nu_{n, m}^{(p)}$. Define
\[
S_{\lambda}^{(m,p)}:=\sum_{T\in\operatorname{SSYT}^{\pm}_{m}(\lambda)}\prod_{b\in\lambda}\operatorname{wt}(T(b)),
\]
where $\operatorname{wt}(j^+)=p/m$ and $\operatorname{wt}(j^-)=q/m$.

\begin{theorem}[Distribution of RSK shape]
\label{thm:rsk-shape}
Let $\pi\sim\nu^{(p)}_{n,m}$. If $Q$ is a standard Young tableau of shape $\lambda\vdash n$, then
\[
\Pbb\{Q(\pi^{-1})=Q\}=S_{\lambda}^{(m,p)}.
\]
Consequently,
\[
\Pbb\{\operatorname{sh}(\pi)=\lambda\}=f^{\lambda}S_{\lambda}^{(m,p)},
\]
where $f^{\lambda}$ is the number of standard Young tableaux of shape $\lambda$.
\end{theorem}

\begin{proof}
By Lemma~\ref{lem:mixed-standardization}, it is enough to work with $\operatorname{std}_{\pm}(a)$.  Apply ordinary RSK to $\operatorname{std}_{\pm}(a)$\textcolor{blue}{\@.} If the entry $r$ of the insertion tableau corresponds to position $i$, replace $r$ by the original letter $a_i$.  The result is a signed semistandard tableau: equal $j^+$ entries form a horizontal strip because their ties were standardized increasingly, and equal $j^-$ entries form a vertical strip because their ties were standardized decreasingly.  Conversely, starting from such a signed semistandard tableau $T$ and a standard tableau $Q$ of the same shape, standardize equal $j^+$ entries in the order compatible with increasing ties and equal $j^-$ entries in the order compatible with decreasing ties, then apply inverse RSK\textcolor{blue}{\@.}  These two constructions are inverse and preserve weights.  Hence, for every fixed recording tableau $Q_0$ of shape $\lambda$, the total weight of words producing it is $S_{\lambda}^{(m,p)}$.  Summing over the $f^{\lambda}$ possible recording tableaux gives the shape formula.
\end{proof}

\begin{remark}
When $p=q=1/2$, the function $S_{\lambda}^{(m,p)}$ becomes the specialization of the extended Schur function used in~\cite[Theorem 3.9]{DiaconisFulmanHolmes2013}. 
\end{remark}

\begin{proposition}[Determinant form]
\label{prop:det-rsk}
Let $h_r^{(m,p)}$ be defined by $h_0^{(m,p)}=1$, $h_r^{(m,p)}=0$ for $r<0$, and
\[
\sum_{r\geq0}h_r^{(m,p)}z^r=\frac{(1+qz/m)^m}{(1-pz/m)^m}.
\]
Then
\[
S_{\lambda}^{(m,p)}=\det\left(h_{\lambda_i-i+j}^{(m,p)}\right)_{1\leq i,j\leq \ell(\lambda)}.
\]
\end{proposition}

\begin{proof}
The generating function for one-row signed semistandard tableaux is the displayed series for $h_r^{(m,p)}$, since positive letters may repeat in a row and negative letters may not.  The usual proof of the Jacobi--Trudi identity~\cite[Theorem 7.16.1]{Stan24}, with the row and column restrictions above, gives the determinant.
\end{proof}

We use Theorem~\ref{thm:rsk-shape} to give a formula for the generating function of $D_n^{\leftarrow}:=\des(\pi^{-1})$, thereby proving the corrected asymmetric analogue of~\cite[Theorem 3.10]{DiaconisFulmanHolmes2013}.

\begin{theorem}[Generating function for inverse descents]
\label{thm:inverse-descent-gf}
Fix $n\geq 1, m\geq 1$ and let $p\in (0, 1)$. Then, 
\[
\E\left[t^{D_n^{\leftarrow}+1}\right]
=(1-t)^{n+1}\sum_{k\geq1}t^k[z^n]
\left(\frac{1+qz/m}{1-pz/m}\right)^{km}.
\]
\end{theorem}

\begin{proof}
For a standard Young tableau $Q$, the descent set of $Q$ is the set of $i$ such that $i+1$ lies in a lower row than $i$.  The standard RSK identity gives $\Des(\sigma)=\Des(Q(\sigma))$ for any permutation $\sigma$~\cite[Lemma 7.23.1]{Stan24}. Let $f_{\lambda}(r)$ be the number of standard Young tableaux of shape $\lambda$ with $r$ descents.  By Theorem~\ref{thm:rsk-shape},
\[
\Pbb\{D_n^{\leftarrow}=r\}=\sum_{\lambda\vdash n}f_{\lambda}(r)S_{\lambda}^{(m,p)}.
\]
The standard tableau descent identity (see~\cite[Equation (7.96)]{Stan24}) gives
\[
\sum_{r\geq0}f_{\lambda}(r)t^{r+1}=(1-t)^{n+1}\sum_{k\geq1}s_{\lambda}(1^k)t^k,
\]
where $s_{\lambda}(1^k)$ is the ordinary Schur function evaluated at $k$ variables all equal to $1$.  The corresponding signed Cauchy identity gives~\cite[Appendix A.4]{Stem97}
\[
\sum_{\lambda}S_{\lambda}^{(m,p)}s_{\lambda}(1^k)z^{|\lambda|}
=\left(\frac{1+qz/m}{1-pz/m}\right)^{km}.
\]
Multiplying the tableau descent identity by $S_{\lambda}^{(m,p)}$, summing over $\lambda\vdash n$, and applying the Cauchy identity proves the theorem.
\end{proof}

We conclude with the following corollary that gives the mean and variance of $D_n^{\leftarrow}$.
\begin{corollary}[Mean and variance of inverse descents]
\label{cor:inverse-desc-moments}
Assume $n\geq 2, m\geq 1$ and $p\in (0, 1)$. Put $\dd=2p-1$. Then,
\[
\E[D_n^{\leftarrow}]=\frac{n-1}{2}-\frac{\dd(n-1)}{2m},
\]
and
\[
\Var(D_n^{\leftarrow})=\frac{n+1}{12}+\frac{n-2}{6m^2}-\frac{\dd^2(n-1)}{4m^2}.
\]
In particular, for $p=q=1/2$, this gives
\[
\E[D_n^{\leftarrow}]=\frac{n-1}{2},
\qquad
\Var(D_n^{\leftarrow})=\frac{n+1}{12}+\frac{n-2}{6m^2},
\]
which recovers~\cite[Corollary 3.11]{DiaconisFulmanHolmes2013}.
\end{corollary}

\begin{proof}
Put
\[
H(z):=\left(\frac{1+qz/m}{1-pz/m}\right)^m,
\qquad
\alpha:=\frac{\dd}{2},
\qquad
\beta:=\frac{p^3+q^3}{3}.
\]
Then
\[
\log H(z)=z+\alpha\frac{z^2}{m}+\beta\frac{z^3}{m^2}+O(z^4).
\]
For fixed $n$, the coefficient $[z^n]H(z)^k$ is a polynomial in $k$ of degree $n$.  Its three highest-degree terms are
\[
[z^n]H(z)^k
=\frac{k^n}{n!}+\frac{\alpha k^{n-1}}{m(n-2)!}
+\left(\frac{\beta}{m^2(n-3)!}+\frac{\alpha^2}{2m^2(n-4)!}\right)k^{n-2}+P_{\leq n-3}(k),
\]
where terms with negative factorials are omitted.  Let
\[
A_j(t):=\sum_{\sigma\in S_j}t^{\des(\sigma)+1}=(1-t)^{j+1}\sum_{k\geq1}k^j t^k
\]
be the Eulerian polynomial.  Substituting the preceding expansion into Theorem~\ref{thm:inverse-descent-gf} gives, up to a term divisible by $(1-t)^3$ (which does not affect the first two derivatives at $t=1$),
\[
\E\left[t^{D_n^{\leftarrow}+1}\right]
=\frac{A_n(t)}{n!}+\frac{\alpha(1-t)A_{n-1}(t)}{m(n-2)!}
+\left(\frac{\beta}{m^2(n-3)!}+\frac{\alpha^2}{2m^2(n-4)!}\right)(1-t)^2A_{n-2}(t).
\]
Using the well-known Eulerian moments (see, for instance, the proof of~\cite[Corollary 3.11]{DiaconisFulmanHolmes2013})
\[
\frac{A'_n(1)}{n!}=\frac{n+1}{2},
\qquad
\frac{A''_n(1)}{n!}=\frac{3n^2+n-2}{12},
\]
we obtain
\[
\E[D_n^{\leftarrow}+1]=\frac{n+1}{2}-\frac{\alpha(n-1)}{m}
\]
and
\[
\E[D_n^{\leftarrow}(D_n^{\leftarrow}+1)]
=\frac{3n^2+n-2}{12}-\frac{\alpha n(n-1)}{m}
+\frac{2\beta(n-2)+\alpha^2(n-2)(n-3)}{m^2}.
\]
The stated variance follows from
\[
\Var(D_n^{\leftarrow})=\E[D_n^{\leftarrow}(D_n^{\leftarrow}+1)]-\E[D_n^{\leftarrow}]-(\E[D_n^{\leftarrow}])^2
\]
and the identity $\beta=(1+3\dd^2)/12$.
\end{proof}


\bibliographystyle{amsalpha} 
\bibliography{references}

@article {chen2025cutoff,
    AUTHOR = {Chen, Ray and Ottolini, Andrea},
     TITLE = {Cutoff in total variation for the shelf shuffle},
   JOURNAL = {Electron. Commun. Probab.},
  FJOURNAL = {Electronic Communications in Probability},
    VOLUME = {30},
      YEAR = {2025},
     PAGES = {Paper No. 44, 10},
      ISSN = {1083-589X},
   MRCLASS = {60J10},
  MRNUMBER = {4908782},
       DOI = {10.1214/25-ecp691},
       URL = {https://doi.org/10.1214/25-ecp691},
}

@unpublished{kuba,
  title={On Card guessing after a single shelf shuffle},
  author={Clay, Alexander  and Kuba, Markus and Tripathi, Raghavendra},
  year={2026},
  note={arXiv:2602.12928},
  publisher={arXiv}
}

@unpublished{clay2025guessing,
  title={Guessing Strategies for Shuffling Machines},
  author={Clay, Alexander},
  journal={arXiv preprint},
  note={arXiv:2507.10294},
  year={2025}
}

@unpublished{clay2025limit,
  title={Limit Theorems for Descents and Inversions of Shelf-Shuffles},
  author={Clay, Alexander},
  journal={arXiv preprint},
  note={arXiv:2510.00343},
  year={2025}
}

@unpublished{tripathi2026position,
  title={On the position matrix of single-shelf shuffle and card guessing},
  author={Tripathi, Raghavendra},
  journal={arXiv preprint},
  note={arXiv:2602.07920},
  year={2026}
}

@unpublished{Trip26Law,
  title={Cutoff for the asymmetric shelf shuffle},
  author={Tripathi, Raghavendra},
  journal={arXiv preprint},
  note={to appear},
  year={2026}
}

@article {Tanny73Eulerian,
    AUTHOR = {Tanny, S.},
     TITLE = {A probabilistic interpretation of {E}ulerian numbers},
   JOURNAL = {Duke Math. J.},
  FJOURNAL = {Duke Mathematical Journal},
    VOLUME = {40},
      YEAR = {1973},
     PAGES = {717--722},
      ISSN = {0012-7094,1547-7398},
   MRCLASS = {05A99 (60C05)},
  MRNUMBER = {340045},
MRREVIEWER = {Bernard\ Harris},
       URL = {http://projecteuclid.org/euclid.dmj/1077310048},
}

@article {DiaconisFulmanHolmes2013,
    AUTHOR = {Diaconis, Persi and Fulman, Jason and Holmes, Susan},
     TITLE = {Analysis of casino shelf shuffling machines},
   JOURNAL = {Ann. Appl. Probab.},
  FJOURNAL = {The Annals of Applied Probability},
    VOLUME = {23},
      YEAR = {2013},
    NUMBER = {4},
     PAGES = {1692--1720},
      ISSN = {1050-5164,2168-8737},
   MRCLASS = {60C05 (05A15)},
  MRNUMBER = {3098446},
MRREVIEWER = {Kent\ E.\ Morrison},
       DOI = {10.1214/12-aap884},
       URL = {https://doi.org/10.1214/12-aap884},
}

@book {DiaconisFulman2023Shuffling,
    AUTHOR = {Diaconis, Persi and Fulman, Jason},
     TITLE = {The mathematics of shuffling cards},
 PUBLISHER = {American Mathematical Society, Providence, RI},
      YEAR = {2023},
     PAGES = {xii+346},
      ISBN = {[9781470463038]},
   MRCLASS = {60-02 (05Axx 60B15 60J10 91A60)},
  MRNUMBER = {4565368},
MRREVIEWER = {Martin\ V.\ Hildebrand},
}

@article {bayer1992trailing,
    AUTHOR = {Bayer, Dave and Diaconis, Persi},
     TITLE = {Trailing the dovetail shuffle to its lair},
   JOURNAL = {Ann. Appl. Probab.},
  FJOURNAL = {The Annals of Applied Probability},
    VOLUME = {2},
      YEAR = {1992},
    NUMBER = {2},
     PAGES = {294--313},
      ISSN = {1050-5164,2168-8737},
   MRCLASS = {60C05 (20B30 60B15)},
  MRNUMBER = {1161056},
MRREVIEWER = {David\ J.\ Aldous},
       URL =
              {http://links.jstor.org/sici?sici=1050-5164(199205)2:2<294:TTDSTI>2.0.CO;2-F&origin=MSN},
}

@article {aldous1986shuffling,
    AUTHOR = {Aldous, David and Diaconis, Persi},
     TITLE = {Shuffling cards and stopping times},
   JOURNAL = {Amer. Math. Monthly},
  FJOURNAL = {American Mathematical Monthly},
    VOLUME = {93},
      YEAR = {1986},
    NUMBER = {5},
     PAGES = {333--348},
      ISSN = {0002-9890,1930-0972},
   MRCLASS = {60C05 (60B15 60G40 60J15)},
  MRNUMBER = {841111},
MRREVIEWER = {Endre\ Cs\'aki},
       DOI = {10.2307/2323590},
       URL = {https://doi.org/10.2307/2323590},
}

@unpublished{salez2025modern,
  title={Modern aspects of Markov chains: entropy, curvature and the cutoff phenomenon},
  author={Salez, Justin},
  journal={arXiv preprint},
  note={arXiv:2508.21055},
  year={2025}
}

@article {GZ20,
    AUTHOR = {Gessel, Ira M. and Zhuang, Yan},
     TITLE = {Counting permutations by peaks, descents, and cycle type},
   JOURNAL = {S\'em. Lothar. Combin.},
  FJOURNAL = {S\'eminaire Lotharingien de Combinatoire},
    VOLUME = {84B},
      YEAR = {2020},
     PAGES = {Art. 12, 12},
      ISSN = {1286-4889},
   MRCLASS = {05A05 (05A15)},
  MRNUMBER = {4138640},
}

@article {GS20Plethystic,
    AUTHOR = {Gessel, Ira M. and Zhuang, Yan},
     TITLE = {Plethystic formulas for permutation enumeration},
   JOURNAL = {Adv. Math.},
  FJOURNAL = {Advances in Mathematics},
    VOLUME = {375},
      YEAR = {2020},
     PAGES = {107370, 55},
      ISSN = {0001-8708,1090-2082},
   MRCLASS = {05A15 (05A05 05E05)},
  MRNUMBER = {4136604},
MRREVIEWER = {Eric\ S.\ Egge},
       DOI = {10.1016/j.aim.2020.107370},
       URL = {https://doi.org/10.1016/j.aim.2020.107370},
}

@article {Hadamard,
    AUTHOR = {Hadamard, Jacques},
     TITLE = {Book {R}eview: {E}lementary {P}rinciples in {S}tatistical
              {M}echanics, {D}eveloped with especial {R}eference to the
              {R}ational {F}oundations of {T}hermodynamics},
   JOURNAL = {Bull. Amer. Math. Soc.},
  FJOURNAL = {Bulletin of the American Mathematical Society},
    VOLUME = {12},
      YEAR = {1906},
    NUMBER = {4},
     PAGES = {194--210},
      ISSN = {0002-9904},
   MRCLASS = {99-04},
  MRNUMBER = {1558318},
       DOI = {10.1090/S0002-9904-1906-01319-2},
       URL = {https://doi.org/10.1090/S0002-9904-1906-01319-2},
}

@book{Poincare1912,
	author = {Henri Poincar\'e},
	editor = {},
	publisher = {Gauthier-Villars},
	title = {Calcul des Probabilit\'{e}s},
	year = {1912}
}

@article {Diaconis81Generating,
    AUTHOR = {Diaconis, Persi and Shahshahani, Mehrdad},
     TITLE = {Generating a random permutation with random transpositions},
   JOURNAL = {Z. Wahrsch. Verw. Gebiete},
  FJOURNAL = {Zeitschrift f\"ur Wahrscheinlichkeitstheorie und Verwandte
              Gebiete},
    VOLUME = {57},
      YEAR = {1981},
    NUMBER = {2},
     PAGES = {159--179},
      ISSN = {0044-3719},
   MRCLASS = {60C05 (60J15)},
  MRNUMBER = {626813},
MRREVIEWER = {Lars\ Holst},
       DOI = {10.1007/BF00535487},
       URL = {https://doi.org/10.1007/BF00535487},
}

@incollection {Aldous83,
    AUTHOR = {Aldous, David},
     TITLE = {Random walks on finite groups and rapidly mixing {M}arkov
              chains},
 BOOKTITLE = {Seminar on probability, {XVII}},
    SERIES = {Lecture Notes in Math.},
    VOLUME = {986},
     PAGES = {243--297},
 PUBLISHER = {Springer, Berlin},
      YEAR = {1983},
      ISBN = {3-540-12289-3},
   MRCLASS = {60J15 (60B15)},
  MRNUMBER = {770418},
MRREVIEWER = {A.\ Mukherjea},
       DOI = {10.1007/BFb0068322},
       URL = {https://doi.org/10.1007/BFb0068322},
}

@article {diaconis1995riffle,
    AUTHOR = {Diaconis, Persi and McGrath, Michael and Pitman, Jim},
     TITLE = {Riffle shuffles, cycles, and descents},
   JOURNAL = {Combinatorica},
  FJOURNAL = {Combinatorica. An International Journal on Combinatorics and the Theory of Computing},
    VOLUME = {15},
      YEAR = {1995},
    NUMBER = {1},
     PAGES = {11--29},
      ISSN = {0209-9683},
   MRCLASS = {05A15},
  MRNUMBER = {1325269},
MRREVIEWER = {Ira\ Gessel},
       DOI = {10.1007/BF01294457},
       URL = {https://doi.org/10.1007/BF01294457},
}

@book {Diaconis88Group,
    AUTHOR = {Diaconis, Persi},
     TITLE = {Group representations in probability and statistics},
    SERIES = {Institute of Mathematical Statistics Lecture Notes---Monograph
              Series},
    VOLUME = {11},
 PUBLISHER = {Institute of Mathematical Statistics, Hayward, CA},
      YEAR = {1988},
     PAGES = {vi+198},
      ISBN = {0-940600-14-5},
   MRCLASS = {60-02 (20C99 62-02)},
  MRNUMBER = {964069},
MRREVIEWER = {Philippe\ Bougerol},
}

@article {Nestoridi25Cutoff,
    AUTHOR = {Nestoridi, Evita and Yan, Alan},
     TITLE = {Cutoff for the biased random transposition shuffle},
   JOURNAL = {S\'em. Lothar. Combin.},
  FJOURNAL = {S\'eminaire Lotharingien de Combinatoire},
    VOLUME = {93B},
      YEAR = {2025},
     PAGES = {Art. 6, 12},
      ISSN = {1286-4889},
   MRCLASS = {60J10 (05E10)},
  MRNUMBER = {4968505},
}

@article {Sellke22Cutoff,
    AUTHOR = {Sellke, Mark},
     TITLE = {Cutoff for the asymmetric riffle shuffle},
   JOURNAL = {Ann. Probab.},
  FJOURNAL = {The Annals of Probability},
    VOLUME = {50},
      YEAR = {2022},
    NUMBER = {6},
     PAGES = {2244--2287},
      ISSN = {0091-1798,2168-894X},
   MRCLASS = {60C05 (20P05 60B15 60J10)},
  MRNUMBER = {4499839},
MRREVIEWER = {Wojciech\ Bartoszek},
       DOI = {10.1214/22-aop1582},
       URL = {https://doi.org/10.1214/22-aop1582},
}

@article {Knuth,
    AUTHOR = {Knuth, Donald E.},
     TITLE = {Johann {F}aulhaber and sums of powers},
   JOURNAL = {Math. Comp.},
  FJOURNAL = {Mathematics of Computation},
    VOLUME = {61},
      YEAR = {1993},
    NUMBER = {203},
     PAGES = {277--294},
      ISSN = {0025-5718,1088-6842},
   MRCLASS = {11B83 (01A45 01A55 11B57)},
  MRNUMBER = {1197512},
MRREVIEWER = {Wells\ Johnson},
       DOI = {10.2307/2152953},
       URL = {https://doi.org/10.2307/2152953},
}

@article {Kellner,
    AUTHOR = {Kellner, Bernd C.},
     TITLE = {Faulhaber polynomials and reciprocal {B}ernoulli polynomials},
   JOURNAL = {Rocky Mountain J. Math.},
  FJOURNAL = {The Rocky Mountain Journal of Mathematics},
    VOLUME = {53},
      YEAR = {2023},
    NUMBER = {1},
     PAGES = {119--151},
      ISSN = {0035-7596,1945-3795},
   MRCLASS = {11B68},
  MRNUMBER = {4585984},
       DOI = {10.1216/rmj.2023.53.119},
       URL = {https://doi.org/10.1216/rmj.2023.53.119},
}

@book {Stan24,
    AUTHOR = {Stanley, Richard P.},
     TITLE = {Enumerative combinatorics. {V}ol. 2},
    SERIES = {Cambridge Studies in Advanced Mathematics},
    VOLUME = {208},
   EDITION = {Second},
      NOTE = {With an appendix by Sergey Fomin},
 PUBLISHER = {Cambridge University Press, Cambridge},
      YEAR = {[2024] \copyright 2024},
     PAGES = {xvi+783},
      ISBN = {978-1-009-26249-1; 978-1-009-26248-4},
   MRCLASS = {05-02 (05A15 05E05 05E10 68R05)},
  MRNUMBER = {4621625},
MRREVIEWER = {Timothy\ Y.\ Chow},
}

@article {Stem97,
    AUTHOR = {Stembridge, John R.},
     TITLE = {Enriched {$P$}-partitions},
   JOURNAL = {Trans. Amer. Math. Soc.},
  FJOURNAL = {Transactions of the American Mathematical Society},
    VOLUME = {349},
      YEAR = {1997},
    NUMBER = {2},
     PAGES = {763--788},
      ISSN = {0002-9947,1088-6850},
   MRCLASS = {06A07 (05E05 05E10)},
  MRNUMBER = {1389788},
MRREVIEWER = {Joseph\ Neggers},
       DOI = {10.1090/S0002-9947-97-01804-7},
       URL = {https://doi.org/10.1090/S0002-9947-97-01804-7},
}

@book {Kyle15,
    AUTHOR = {Petersen, T. Kyle},
     TITLE = {Eulerian numbers},
    SERIES = {Birkh\"auser Advanced Texts: Basler Lehrb\"ucher.
              [Birkh\"auser Advanced Texts: Basel Textbooks]},
      NOTE = {With a foreword by Richard Stanley},
 PUBLISHER = {Birkh\"auser/Springer, New York},
      YEAR = {2015},
     PAGES = {xviii+456},
      ISBN = {978-1-4939-3090-6; 978-1-4939-3091-3},
   MRCLASS = {05-02 (05A15 05Exx 06A07 11B65 11B75 20F55)},
  MRNUMBER = {3408615},
MRREVIEWER = {Damir\ Yeliussizov},
       DOI = {10.1007/978-1-4939-3091-3},
       URL = {https://doi.org/10.1007/978-1-4939-3091-3},
}

@book {Comtet74,
    AUTHOR = {Comtet, Louis},
     TITLE = {Advanced combinatorics},
   EDITION = {enlarged},
      NOTE = {The art of finite and infinite expansions},
 PUBLISHER = {D. Reidel Publishing Co., Dordrecht},
      YEAR = {1974},
     PAGES = {xi+343},
      ISBN = {90-277-0441-4},
   MRCLASS = {05-02},
  MRNUMBER = {460128},
}

@book {Gasper04,
    AUTHOR = {Gasper, George and Rahman, Mizan},
     TITLE = {Basic hypergeometric series},
    SERIES = {Encyclopedia of Mathematics and its Applications},
    VOLUME = {96},
   EDITION = {Second},
      NOTE = {With a foreword by Richard Askey},
 PUBLISHER = {Cambridge University Press, Cambridge},
      YEAR = {2004},
     PAGES = {xxvi+428},
      ISBN = {0-521-83357-4},
   MRCLASS = {33Dxx (05A30 05E35 33-01 33-02)},
  MRNUMBER = {2128719},
MRREVIEWER = {Shaun\ Cooper},
       DOI = {10.1017/CBO9780511526251},
       URL = {https://doi.org/10.1017/CBO9780511526251},
}

\end{document}